\documentclass[twocolumns]{autart}

\usepackage{setspace}

\usepackage{url}
\usepackage[numbers]{natbib}            
\usepackage{graphicx}          
\usepackage{slashbox}                               

\usepackage{amssymb}
\usepackage{amsmath}
\usepackage{graphicx}
\usepackage{amsfonts}
\usepackage{multirow}
\usepackage{color}
\usepackage{psfrag}
\usepackage{mathrsfs}
\usepackage{dsfont}
\usepackage{etoolbox}

\renewcommand{\theenumi}{\alph{enumi}}

\everymath{\displaystyle}

\newtheorem{theorem}{Theorem}[section]

\newenvironment{proof}{{\it Proof :~}}{\hfill $\diamondsuit$}
\newtheorem{remark}{Remark}
\newtheorem{example}{Example}

\def\d{\textnormal{d}}
\DeclareMathOperator*{\He}{Sym}

\DeclareMathOperator{\eps}{\varepsilon}

\DeclareMathOperator{\co}{\mathbf{co}}

\AtBeginEnvironment{bmatrix}{\setlength{\arraycolsep}{5pt}}

\begin{document}

\begin{frontmatter}

\title{Convex conditions for robust stability analysis and stabilization of linear aperiodic impulsive and sampled-data systems under dwell-time constraints}

\author{Corentin Briat}\ead{briatc@bsse.ethz.ch,corentin@briat.info}\ead[url]{http://www.briat.info}

\address{Swiss Federal Institute of Technology--Z\"{u}rich (ETH-Z), Department of Biosystems Science and Engineering (D-BSSE), Switzerland.}

%


\begin{abstract}
Stability analysis and control of linear impulsive systems is addressed in a hybrid framework, through the use of continuous-time time-varying discontinuous Lyapunov functions. Necessary and sufficient conditions for stability of impulsive systems with periodic impulses are first provided in order to set up the main ideas. Extensions to stability of aperiodic systems under minimum, maximum and ranged dwell-times are then derived. By exploiting further the particular structure of the stability conditions, the results are non-conservatively extended to quadratic stability analysis of linear uncertain impulsive systems. These stability criteria are, in turn, losslessly extended to stabilization using a particular, yet broad enough, class of state-feedback controllers, providing then a convex solution to the open problem of robust  dwell-time stabilization of impulsive systems using hybrid stability criteria. Relying finally on the representability of sampled-data systems as impulsive systems, the problems of robust stability analysis and robust stabilization of periodic and aperiodic uncertain sampled-data systems are straightforwardly solved using the same ideas. Several examples are discussed in order to show the effectiveness and reduced complexity of the proposed approach.
\end{abstract}

\end{frontmatter}

\section{Introduction}

Impulsive systems \citep{Bainov:89, Hespanha:08, Michel:08,Goebel:09, Briat:11l} are an important class of hybrid systems exhibiting both continuous- and discrete-time dynamics. The discrete-time part, which is only active at certain time instants $t_k$, $k\in\mathbb{N}$, introduces discontinuities in the overall trajectories of the system.  
Analyzing them usually relies on the use of Lyapunov functions and input-to-state stability/nonlinear small-gain ideas \citep{Hespanha:08,Nesic:11,Dashkovskiy:12}, Lyapunov functionals \cite{Naghshtabrizi:08} or, more recently, another type of functionals, verifying certain boundary conditions, referred to as \emph{looped-functionals} \citep{Briat:11l, Briat:13b}. When the impulses occur periodically, the system can be viewed as an LTI discrete-time system which can be studied using discrete-time Lyapunov theory. When impulses arrive at irregular times (aperiodic regime), the discrete-time system becomes time-varying and specific stability concepts should then be considered. The notion of \emph{dwell-time}, i.e. the time between two successive discrete events defined as $T_k:=t_{k+1}-t_k$, has been introduced early in the literature \cite{Morse:96,Hespanha:99} and has been proven to be very useful for the analysis of switched systems. In the case of impulsive systems, dwell-times more specifically correspond to the times between two consecutive impulses. Impulsive systems can therefore be identified through the properties of the sequence of impulse instants $\{t_k\}$, and a relevant stability notion can therefore be considered. When the sequence of impulsive instants is arbitrary, i.e. $T_k>0$, we talk about stability under arbitrary dwell-time, whereas stability under ranged dwell-time is defined for sequences verifying $T_k\in[T_{min},T_{max}]$. Stability under minimum and maximum dwell-time address the cases $T_k\ge\bar{T}$ and $T_k\le\bar{T}$, respectively.

\vspace{-2mm}Stability under dwell-time constraints can be analyzed in several different ways. Lyapunov approaches based on a separate worst-case convergence analysis  (i.e. $\alpha$-stability) of the distinct parts of the impulsive system \cite{Hespanha:99,Hespanha:08} are very convenient to work with when dealing with uncertainties, or when control design is the main goal, principally due to their convexity properties. They may, however, be unable to yield very accurate estimates for dwell-times \cite{Geromel:06b,Briat:13b} since they may not capture the possible interplay between the continuous- and discrete-parts. Discrete-time approaches, however, exhibit much less conservatism, but are, in the present state-of-the-art, difficult to adapt to uncertain systems and aperiodic systems or to extend to control design, mainly due to a lack of convexity. Hybrid stability conditions (also referred to as ``mixed stability conditions'' in the following) consisting of coupled continuous-time and discrete-time criteria have been shown to yield more accurate estimates for minimum dwell-time for both linear switched systems \cite{Geromel:06b,Chesi:12} and linear impulsive systems \cite{Briat:11l,Briat:12h}. The main difficulties when considering hybrid conditions lie in the nonconvex dependence on the system matrices (due to the presence of a discrete-time condition), that complicates the extensions to both time-invariant and time-varying uncertain systems, and to control design. 

\vspace{-2mm}Periodic and aperiodic sampled-data systems, arising for instance in digital control \citep{Chen:95} or networked control systems \citep{Hespanha:07}, are intimately connected to impulsive systems since any sampled-data system can be equivalently represented as an impulsive system. Several approaches have been developed to analyze sampled-data systems: discrete-time approaches
\citep{Fujioka:09a,Oishi:10,Cloosterman:10,Donkers:11}, input-delay approaches  \citep{Teel:98b,Fridman:04,Fridman:10}, robust analysis techniques \citep{Mirkin:07,Fujioka:09b,Kao:13}, impulsive/hybrid systems formulation \citep{Sun:91,Sivashankar:94,Dai:10,Briat:11l,Briat:12h}, and the use of looped-functionals either considering directly the sampled-data system formulation \citep{Seuret:12} or the impulsive system formulation \citep{Briat:11l,Briat:12h}.
These approaches have exactly the same benefits and drawbacks as in the case of impulsive systems.

\vspace{-2mm}The rationale for using mixed stability criteria for analyzing switched and impulsive systems \citep{Geromel:06b,Chesi:12,Briat:11l,Briat:12h,Briat:13b} lies in the reduced (possibly vanishing) conservatism \cite{Wirth:05} of the conditions, opposed to continuous-time results based on rates of convergence of Lyapunov functions, see e.g. \citep{Morse:96,Hespanha:99,Hespanha:08}. Hybrid stability criteria are therefore important to consider in order guarantee accuracy, but should be characterized in such a way that \emph{robustness analysis and control design remain possible}. A first step towards such a result has been made very recently by using looped-functionals \citep{Briat:11l,Briat:12h,Briat:13b}, a specific type of functionals defined on a lifted state-space which encode a discrete-time condition as a convex condition in the system matrices, a very suitable feature for robust stability analysis. However, the structure of the conditions prevents the derivation of tractable design criteria due to the presence of multiple decision matrices, inexorably leading to high computational cost and nonconvex terms in the synthesis conditions. The proposed approach, based on time-varying continuous-time discontinuous Lyapunov functions, combines features of the continuous-time and hybrid approaches  by leading to necessary and sufficient stability and stabilization conditions which are convex in the system matrices and in the decision variables (Lyapunov and controller variables), together with a lower complexity than by using looped-functionals.

\vspace{-2mm}The contribution of this paper lies on different levels. First of all, necessary and sufficient conditions for stability of impulsive systems with periodic impulses are derived in Section \ref{sec:stab_is} from the use of a specific discontinuous Lyapunov function. The advantage of the use of such Lyapunov functions lies in a reduced computational complexity over the use of looped-functional, while accuracy is mostly preserved. The periodic case is then extended to cope with aperiodicity in impulse arrival times (i.e. minimum, maximum, and ranged dwell-times) and time-varying parametric uncertainties. Necessary and sufficient results for discrete-time quadratic stability are provided, again with a reduced computational complexity. By relying on non-conservative algebraic manipulations, these results are further exactly adapted in Section \ref{sec:stabz_is} to quadratic (robust) stabilization using a particular class of state-feedback controllers.  More concisely, quadratic stabilization with prescribed minimum, maximum or ranged dwell-times can be expressed as convex optimization problems. The approach is fully generic and can be applied to any linear impulsive system. Exploiting finally, in Section \ref{sec:sds}, the representability of sampled-data systems as impulsive systems, the results are then adapted to sampled-data systems. Convex necessary and sufficient conditions for quadratic stabilization of aperiodic uncertain time-varying sampled-data systems are obtained. Examples and comparisons with several existing results are discussed in the related sections.


\vspace{-2mm}\textbf{Notations:} The set of $n\times n$ (positive definite) symmetric matrices is denoted by ($\mathbb{S}_{\succ0}^n$) $\mathbb{S}^n$. Given two symmetric matrices $A,B$, the inequality $A\succ(\succeq) B$ means that $A-B$ is positive (semi)definite. Given a square real matrix $A$, the notation $\He(A)$ stands for the sum $A+A^T$.

\section{Stability analysis of periodic and aperiodic impulsive systems}\label{sec:stab_is}

In this section, linear impulsive systems of the form
\begin{equation}\label{eq:impsyst}
  \begin{array}{lcl}
    \dot{x}(t)&=&Ax(t),\ t\ne t_k\\
    x(t)&=&Jx^-(t),\ t=t_k
  \end{array}
\end{equation}
are considered where $x\in\mathbb{R}^n$ is the state of the system and $x^-(t_k)$ stands for the left-limit of $x(s)$ at $s=t_k$, i.e. $\textstyle{x^-(t_k)=\lim_{s\uparrow t_k}x(s)}$. The system matrices $A$ and $J$ may be uncertain time-varying, this will be explicitly mentioned when this is the case. The sequence of impulse instants $\{t_k\}_{k\in\mathbb{N}}$, $t_k>0$, is assumed to have positive increments $T_k:=t_{k+1}-t_k>\epsilon>0$ that are bounded away from 0. Defined as such, the sequence $\{t_k\}_{k\in\mathbb{N}}$ does not admit any accumulation point (we exclude then any Zeno motion) and grows unbounded. Note that the sequence of impulse instants may or may not depend on the state of the system. In the following, we will make no distinction between these two cases since impulse sequences will be solely characterized in terms of dwell-time constraints.


\subsection{Impulsive systems with periodic impulses}


The case of periodic impulses is addressed fist in order to introduce the main ideas.

%
\begin{theorem}[Periodic impulses]\label{th:imp_p}
 Let us consider the system \eqref{eq:impsyst} with periodic impulses, i.e. $T_k=\bar{T}$, $k\in\mathbb{N}$. Then, the following statements are equivalent:
 \begin{enumerate}
   \item The impulsive system (\ref{eq:impsyst}) with $\bar{T}$-periodic impulses is asymptotically stable.
   \item The discrete-time transition matrix $\Psi(\bar{T}):=e^{A\bar{T}}J$  is Schur\footnote{A matrix is Schur (or Schur stable) if all its eigenvalues lie in the unit disc.}.
   \item There exists a matrix $P\in\mathbb{S}_{\succ0}^n$ such that the LMI
    \begin{equation}\label{eq:stabmono}
       J^Te^{A^T\bar{T}}Pe^{A\bar{T}}J-P\prec0
    \end{equation}
    holds or, equivalently, the quadratic form $V(x)=x^TPx$ is a discrete-time Lyapunov function for the LTI discrete-time system $z_{k+1}=e^{A\bar{T}}Jz_k$.
   %
    %
   \item There exist a differentiable matrix function $R:[0,\bar{T}]\mapsto\mathbb{S}^n$, $R(0)\succ0$, and a scalar $\eps>0$ such that the LMIs
  \begin{equation}\label{eq:c1}
    A^TR(\tau)+R(\tau)A+\dot{R}(\tau)\preceq0
  \end{equation}
  and
  \begin{equation}\label{eq:c2}
    J^TR(0)J-R(\bar{T})+\eps I\preceq0
  \end{equation}
  hold for all $\tau\in[0,\bar{T}]$.
   \item There exist a differentiable matrix function $S:[0,\bar{T}]\mapsto\mathbb{S}^n$, $S(\bar{T})\succ0$, and a scalar $\eps>0$ such that the LMIs
  \begin{equation}\label{eq:c1b}
    A^TS(\tau)+S(\tau)A-\dot{S}(\tau)\preceq0
  \end{equation}
  and
  \begin{equation}\label{eq:c2b}
    J^TS(\bar{T})J-S(0)+\eps I\preceq0
  \end{equation}
  hold for all $\tau\in[0,\bar{T}]$.
 \end{enumerate}
\end{theorem}
\begin{proof}
The proof that (a) $\Leftrightarrow$ (b) $\Leftrightarrow$ (c) can be found in \cite{Briat:11l}.

\textbf{Proof of (d) $\Rightarrow$ (c):} Assume (d) holds. Integrating \eqref{eq:c1} over $[0,\bar{T}]$, pre-  and post-multiplying by $J^T$ and $J$ implies that the LMIs
\begin{equation}\label{eq:djksds}
  J^Te^{A^T\bar{T}}R(\bar{T})e^{A\bar{T}}J-J^TR(0)J\preceq0
\end{equation}
holds. From \eqref{eq:c2}, we have that $R(\bar{T})\succ0$ and $J^TR(0)J\preceq R(\bar{T})-\eps I$. Substituting then for $J^TR(0)J$ in \eqref{eq:djksds} yields
\begin{equation}
  J^Te^{A^T\bar{T}}R(\bar{T})e^{A\bar{T}}J-R(\bar{T})\preceq-\eps I
\end{equation}
which therefore implies that \eqref{eq:stabmono} holds with $P=R(\bar{T})\succ0$. The proof is complete.

\textbf{Proof of (c) $\Rightarrow$ (d):} The proof is structured as follows: first, we prove that \eqref{eq:c1} admits solutions regardless of the stability of the system, showing that this condition can be assumed to be satisfied without loss of generality. The second part of the proof consists of combining statement (c) with the solution set of \eqref{eq:c1} to prove that \eqref{eq:c2} holds.

Assume \eqref{eq:stabmono} holds with $P=R(\bar{T})\succ0$ and some $Y\succ0$ as
\begin{equation}\label{eq:djqsodjsopd}
       J^Te^{A^T\bar{T}}R(\bar{T})e^{A\bar{T}}J-R(\bar{T})=-Y.
\end{equation}
Since $e^{A\bar{T}}J$ is Schur, then the above matrix equation admits a unique solution $R(\bar{T})\succ0$ \citep{Gahinet:90}.
%
%
The set of all solutions $R(\tau)$ to \eqref{eq:c1} can be defined as the set of solutions of the matrix equality
\begin{equation}\label{eq:dkjsodj}
   A^TR(\tau)+R(\tau)A+\dot{R}(\tau)=-W(\tau),\ W(\tau)\succeq0
\end{equation}
where $W(\tau)$ is a continuous function w.l.o.g. Given $W(\tau)$, the unique solution to \eqref{eq:dkjsodj} is given by
\begin{equation}\label{eq:sol}
\begin{array}{lcl}
    R(\tau)&=&e^{-A^T\tau}R(0)e^{-A\tau}\\
    &&-\int_0^\tau e^{-A^T(\tau-s)}W(s)e^{-A(\tau-s)}\d s,\ \tau\in[0,\bar{T}]
\end{array}
\end{equation}
where $R(\bar{T})\succ0$ is defined by \eqref{eq:djqsodjsopd}. We have proved that \eqref{eq:dkjsodj} can be considered as fulfilled, independently of the stability of the system, which concludes the first part of the proof.
The second part of the proof consists of deriving first, from expression \eqref{eq:sol}, the equation
\begin{equation}\label{eq:kdpsdo}
  e^{A^T\bar{T}}R(\bar{T})e^{A\bar{T}}=-\tilde{W}(\bar{T})+R(0)
\end{equation}
where $\textstyle\tilde{W}(\bar{T})=\int_0^{\bar{T}} e^{A^Ts}W(s)e^{As}\d s\succeq0$. The above equality implies that $R(0)\succ\tilde{W}(\bar{T})\succeq0$ since $R(\bar{T})\succ0$. Consequently, we have that $R(\tau)\succ0$ for all $\tau\in[0,\bar{T}]$ since $\tilde{W}(\tau)\succeq0$ is a nondecreasing function, i.e. $\tilde{W}(\tau)\preceq \tilde{W}(\zeta)$ for any $0\le\tau\le\zeta\le\bar{T}$. Substituting, finally, the left-hand side of \eqref{eq:kdpsdo} in \eqref{eq:djqsodjsopd}, we get that $J^TR(0)J-R(\bar{T})=-Y+J^T\tilde{W}(\bar{T})J$. Since $W(s)$ and $Y\succ0$ are arbitrary, then we can choose them such that this then implies that $-Y+J^T\tilde{W}(\bar{T})J\prec0$ and thus that \eqref{eq:c2} holds. The proof is complete.

\textbf{Proof of (d) $\Leftrightarrow$ (e):} Assume  (d) holds for some $R(\tau)$, it is immediate to see that ${R(\tau):=S(\bar{T}-\tau)}$ solves \eqref{eq:c1b} and \eqref{eq:c2b}. Reverting the argument proves the equivalence.
\end{proof}

The conditions stated in statement (d) can be understood as a non-increase condition, over each interval $[t_k,t_{k+1})$, of the time-varying discontinuous Lyapunov function
     $V_d(x,\tau)=x^T\hat{Q}(\tau)x$
   where $\hat{Q}(t_k+\tau)=Q(\tau)$, $Q(\tau)\in\mathbb{S}^n$, $\tau\in[0,T_k)$, $Q(0)\succ0$, $Q$ differentiable, and that verifies the boundary condition
   \begin{equation}
      J^TQ(0)J-Q^-(\bar{T})+\eps I\preceq0
   \end{equation}
   where $\textstyle{Q^-(\bar{T})=\lim_{s\uparrow\bar{T}}\{Q(s)\}}$.

A peculiarity of the proposed approach is that the matrices $R(\tau)$ and $S(\tau)$ do not need to be imposed to be positive definite over their domain of definition. Positivity over their domain is directly implied from the positivity of $R(0)$ and $S(\bar{T})$, and the LMI conditions in statements (d) and (e). These conditions, all together, indeed imply that $R(\bar{T})$ is also positive definite, and thus that, by virtue of equation \eqref{eq:sol} that $R(\tau)$ is positive definite on its domain. The case of $S(\tau)$ is symmetric.


There are several advantages of the conditions \eqref{eq:c1}-\eqref{eq:c2} of statement (d) (or conditions \eqref{eq:c1b}-\eqref{eq:c2b} of statement (e)) over condition \eqref{eq:stabmono} of statement (c). First of all, the conditions are convex in the system matrices $A$ and $J$, allowing then for an immediate extension to the uncertain case. Further, the presence of a single decision matrix variable in the conditions tends to suggest the possibility of deriving tractable synthesis conditions. The compensation for these interesting convexity properties is the consideration of infinite dimensional feasibility problems, which may be very hard to solve. Several methods can be applied to make the feasibility problems finite-dimensional. A first one is to discretize the interval $[0,\bar{T}]$ and express the matrix $R(\tau)$ as a piecewise linear function on each subintervals; see e.g. \cite{Allerhand:13}. A second one relies on  sum of squares programming \citep{Parrilo:00} which provides an efficient framework for solving such problems by restricting the matrix functions $R(\tau)$ and $S(\tau)$ to polynomial matrix functions. 
It is also very important to point out that the computation complexity is improved by the fact that $R(\tau)$ does not have to be specifically imposed to be positive definite over $[0,\bar{T}]$ since this is a direct consequence of the conditions $R(0)\succ0$, \eqref{eq:c1} and \eqref{eq:c2} of the theorem.

Still in a computational perspective, it seems necessary to compare the computational complexity of the conditions of Theorem \ref{th:imp_p} to the complexity of the looped-functional-based results of \cite{Briat:12h} addressing the same problem. Assuming polynomial matrices $R(\tau),S(\tau)\in\mathbb{S}^n$ of degree $d_R$ in Theorem \ref{th:imp_p} and a polynomial matrix $Z(\tau)\in\mathbb{S}^{3n}$ of degree $d_Z$ in \citep{Briat:12h}, we have the following count of the number of variables:
\begin{equation}\label{eq:NN}
\begin{array}{lcl}
    N_{current}(d_R)    &=& (d_R+1)\dfrac{n(n+1)}{2}\\
    N_{looped}(d_Z)     &=& \dfrac{n(n+1)}{2}+(d_Z+1)\dfrac{3n(3n+1)}{2}
\end{array}
\end{equation}
for the current approach and the looped-functional approach of \cite{Briat:12h}, respectively. We can immediately see that the number of variables for the looped-functional approach grows much faster with the system dimension $n$ and the degree of the polynomial than with the proposed approach. It seems, however, important to stress that the expressions \eqref{eq:NN} should be understood as lower bounds on the actual computational complexity since additional variables are usually needed, e.g. to incorporate constraints. It will be illustrated in the examples that the proposed approach is able to obtain results that are very close to those obtained using looped-functionals with a much lower computational complexity, even if $d_R$ is usually larger than $d_Z$. A comparison will also be made with a discretization-based approach.



\subsection{Aperiodic impulsive systems}

Let us consider now that the system~\eqref{eq:impsyst} is aperiodic, i.e. impulses arrive at irregular times. To this aim, we consider a ranged-dwell time constraint on the sequence of impulse instants, i.e. $T_k\in[T_{min},T_{max}]$. We then have the following generalization of Theorem \ref{th:imp_p}:
\begin{theorem}[Ranged dwell-time]\label{th:imp_a}
 Let us consider the system \eqref{eq:impsyst} with a ranged dwell-time constraint, i.e. $T_k\in[T_{min},T_{max}]$, $k\in\mathbb{N}$. Then, the following statements are equivalent:
 \begin{enumerate}
    \item There exists a matrix $P\in\mathbb{S}_{\succ0}^n$ such that the LMI
    \begin{equation}\label{eq:AIb}
       J^Te^{A^T\theta}Pe^{A\theta}J-P\prec0
    \end{equation}
    holds for all $\theta\in[T_{min},T_{max}]$.
   \item There exist a differentiable matrix function $R:[0,T_{max}]\mapsto\mathbb{S}^n$, $R(0)\succ0$, and a scalar $\eps>0$ such that the LMIs
  \begin{equation}\label{eq:AIa1}
    A^TR(\tau)+R(\tau)A+\dot{R}(\tau)\preceq0
  \end{equation}
  and
  \begin{equation}\label{eq:AIa2}
    J^TR(0)J-R(\theta)+\eps I\preceq0
  \end{equation}
hold for all $\tau\in[0,T_{max}]$ and all $\theta\in[T_{min},T_{max}]$.
 \end{enumerate}
 Moreover, when one of the above statements holds, then the aperiodic impulsive system~\eqref{eq:impsyst} with ranged dwell-time $T_k\in[T_{min},T_{max}]$ is asymptotically stable.
\end{theorem}
\begin{proof}
  The proof follows the same lines as the one of Theorem \ref{th:imp_p}
\end{proof}
It is important to stress that, in the result above, the computational complexity of the second statement is much lower than if we had used conditions \eqref{eq:c1}-\eqref{eq:c2} of Theorem \ref{th:imp_p}, statement (e). This follows from the fact that in statement (e), we would have required $S(\theta)\succ0$ for all $\theta\in[0,T_{max}]$, which is obviously much more complex than simply imposing $R(0)\succ0$ in the present case. To pursue on the computational complexity analysis, we note that the conditions of Theorem \ref{th:imp_a} are more expensive than those of Theorem \ref{th:imp_p} due to the presence of the additional parameter $\theta\in[0,T_{max}]$ in LMI \eqref{eq:AIa2}. 


%

The next result concerns stability of impulsive system under minimum dwell-time, i.e. $T_k\ge\bar{T}$ for all $k\in\mathbb{N}$. This stability concept has been extensively studied in the past, see. e.g. \citep{Hespanha:08,Briat:11l,Briat:12h} and references therein.
\begin{theorem}[Minimum Dwell-Time]\label{th:imp_dt}
  Let us consider the system \eqref{eq:impsyst} with a minimum dwell-time constraint, i.e. $T_k\ge\bar{T}$, $k\in\mathbb{N}$. Then, the following statements are equivalent:
  \begin{enumerate}
  %
    \item There exists a matrix $P\in\mathbb{S}_{\succ0}^n$ such that the LMIs
    \begin{equation}\label{eq:minDTa1}
      A^TP+PA\prec0
    \end{equation}
    and
    \begin{equation}\label{eq:minDTa2}
          J^Te^{A^T\bar{T}}Pe^{A\bar{T}}J-P\prec0
    \end{equation}
    hold.
    \item There exist a differentiable matrix function $R:[0,\bar{T}]\mapsto\mathbb{S}^n$, $R(0)\succ0$, and a scalar $\eps>0$ such that the LMIs
\begin{equation}\label{eq:minDTb1}
      A^TR(0)+R(0)A\prec0
    \end{equation}
  \begin{equation}\label{eq:minDTb2}
    A^TR(\tau)+R(\tau)A+\dot{R}(\tau)\preceq0
  \end{equation}
  and
  \begin{equation}\label{eq:minDTb3}
    J^TR(0)J-R(\bar{T})+\eps I\preceq0
  \end{equation}
 hold for all $\tau\in[0,\bar{T}]$.
     \item There exist a differentiable matrix function $S:[0,\bar{T}]\mapsto\mathbb{S}^n$, $S(\bar{T})\succ0$, and a scalar $\eps>0$ such that the LMIs
\begin{equation}\label{eq:minDTc1}
      A^TS(\bar{T})+S(\bar{T})A\prec0
    \end{equation}
  \begin{equation}\label{eq:minDTc2}
    A^TS(\tau)+S(\tau)A-\dot{S}(\tau)\preceq0
  \end{equation}
  and
  \begin{equation}\label{eq:minDTc3}
    J^TS(\bar{T})J-S(0)+\eps I\preceq0
  \end{equation}
 hold for all $\tau\in[0,\bar{T}]$.
  \end{enumerate}
  Moreover, when one of the above statements holds, the impulsive system~\eqref{eq:impsyst} is asymptotically stable under minimum dwell-time $\bar{T}$, i.e. for any sequence $\{t_{k}\}_{k\in\mathbb{N}}$ such that $T_k\ge\bar{T}$.
\end{theorem}
\begin{proof}
The proof that  $(a)\Leftrightarrow(b)\Leftrightarrow(c)$ follows from Theorem \ref{th:imp_p}. The proof that $(a)$ implies stability with minimum dwell-time can be found in \citep{Briat:11l,Briat:12h}.
%
%
%
\end{proof}

It is important to note that the above theorem straightforwardly extends to time-varying systems depending explicitly on time and/or time-varying parameters by simply using the fundamental-solution and the state-transition matrices instead of exponentials. The variational argument used to prove the equivalence between statements (a) and (b) remains also valid.

\begin{remark}\label{rk:maxDT}
  Similarly to as in \citep{Briat:11l,Briat:12h}, a maximum dwell-time result can be obtained by simply reverting the inequality sign in the LMIs  \eqref{eq:minDTa1}, \eqref{eq:minDTb1}  and \eqref{eq:minDTc1}. In such a case, the concluding statement changes to: "The aperiodic impulsive system~\eqref{eq:impsyst} is asymptotically stable under maximum dwell-time $\bar{T}$, i.e. for any sequence $\{t_{k}\}_{k\in\mathbb{N}}$ such that $T_k\in[\epsilon,\bar{T}]$ for any $\epsilon>0$."
\end{remark}

\subsection{Examples}

The conditions stated in Theorems \ref{th:imp_p}, \ref{th:imp_a} and \ref{th:imp_dt} are infinite-dimensional feasibility problems. In order to enforce them efficiently, the sum-of-squares programming package SOSTOOLS \citep{sostools} and the semidefinite programming solver SeDuMi \citep{Sturm:01a} are used. Suitable matrix functions $R$ or $S$ such that the conditions of Theorems \ref{th:imp_p}, \ref{th:imp_a} and \ref{th:imp_dt} are feasible are then searched within the set of matrix polynomials of fixed (and chosen) degree, $d_R$ say. In this case, the matrix function $R(\tau)$ is chosen as $\textstyle R(\tau)=\sum_{i=0}^{d_R}R_i\tau^i$, $R_i\in\mathbb{S}^n$, and, in this regard, its derivative is simply given by the polynomial $\textstyle\dot{R}(\tau)=\sum_{i=1}^{d_R}iR_i\tau^{i-1}$ which can be easily inserted in the SOS conditions. In the examples below, the number of variables is identified as the number of variables declared by SOSTOOLS when defining the matrix decision variables, i.e. the Lyapunov matrix $P(\tau)$ and the SOS variables $M_i(\tau)$'s for constraints incorporation. Simulations are performed on an i7-2620M @ 2.70 Ghz with 4GB of RAM.

Note that even though the results obtained in the following examples are compared with the results of \citep{Briat:11l,Briat:12h}, other methods such as the one described in \citep{Dashkovskiy:13} can be applied as well.

\begin{example}[Ranged dwell-time]\label{ex:ranged}
    Let us consider the system~(\ref{eq:impsyst}) with matrices \citep{Briat:11l,Briat:12h}
  \begin{equation}
  \begin{array}{lclclcl}
        A&=&\begin{bmatrix}
       -1 & 0.1\\
       0 & 1.2
    \end{bmatrix},&& J&=&\begin{bmatrix}
      1.2 &0\\
      0 &0.5
    \end{bmatrix}.
  \end{array}
  \end{equation}
  By computing the eigenvalues of $e^{A\bar{T}}J$, this system can be easily shown to be stable with $\bar{T}$-periodic impulses whenever $\bar{T}\in[0.1824, 0.5776]$. Using the ranged dwell-time stability conditions \eqref{eq:AIa1}-\eqref{eq:AIa2} of Theorem \ref{th:imp_a}, the same bounds are retrieved with a matrix polynomial $R$ of order 6, showing then tightness of the obtained numerical values in the aperiodic case. For comparison purposes, the same numerical result is obtained in \cite{Briat:12h} using a looped-functional of degree $d_Z=3$. 
  SOSTOOLS, however, declares 149 variables for the current approach, whereas for looped-functionals $2806$ variables are involved. The execution time is about 1 second whereas it is approximately of 15 seconds for the looped functional approach of \citep{Briat:12h}.
  \end{example}
\begin{table}[ht]
\centering
\begin{tabular}{|c|c||c|c|}
    \hline
      & $d_R$& $T_{min}$ & $T_{max}$\\
      \hline
      \hline
      \multirow{3}{*}{Theorem \ref{th:imp_a}, (b)} & 2 & 0.1834 & 0.4998\\   
      &4 & 0.1824 & 0.5768\\
      &6 & 0.1824 & 0.5776\\
      \hline
      Periodic case & -- & 0.1824 & 0.5776\\
      \hline
    \end{tabular}
    \caption{Estimates of the admissible range of dwell-times for the aperiodic system of Example \ref{ex:ranged}}\label{tab:cranged}
\end{table}
\begin{table}[ht]
\centering
\begin{tabular}{|c|c|c|}
    \hline
      & $d_R$ & $T_{min}$\\       
      \hline
      \hline
      \multirow{3}{*}{Theorem \ref{th:imp_dt}, (c)}&2 & 1.1883\\
      &4 & 1.1408\\
      &6 & 1.1406\\
      \hline
      Theorem \ref{th:imp_dt}, (b) & -- & 1.1406\\
      \hline
      Periodic case & -- & 1.1406\\
      \hline
    \end{tabular}
    \caption{Estimates of the minimum dwell-time for Example \ref{ex:min}}\label{tab:cmin}
\end{table}

\begin{example}[Minimum dwell-time]\label{ex:min}
  Let us consider the system~(\ref{eq:impsyst}) with matrices \citep{Briat:11l,Briat:12h}
  \begin{equation}
  \begin{array}{lclclcl}
        A&=&\begin{bmatrix}
   -1 &0\\
   1 &-2
    \end{bmatrix},& &J&=&\begin{bmatrix}
      2 & 1\\
      1 & 3
    \end{bmatrix}.
  \end{array}
  \end{equation}
  Since $A$ is Hurwitz, the minimum dwell-time result stated in Theorem \ref{th:imp_dt} can be applied. Using conditions \eqref{eq:minDTa1}-\eqref{eq:minDTa2}, we get the minimum dwell-time $\bar{T}=1.1405$. The same value for the minimum dwell-time is obtained using conditions  \eqref{eq:minDTb1}-\eqref{eq:minDTb2}-\eqref{eq:minDTb3} with a polynomial matrix $R$ of order 6; see Table \ref{tab:cmin}.
Using the looped-functional approach of \cite{Briat:12h}, this numerical result is obtained by using polynomials of degree $d_Z=3$ (i.e. 412 variables), whereas the current approach involving polynomials of order 6 only requires 85 variables. The execution time is about 0.5 second whereas it is approximately of 1.5 seconds for the looped functional approach of \citep{Briat:12h}. For comparison, we also consider a discretization scheme \cite{Allerhand:13} where $R(\tau)$ is expressed as a piecewise linear function on $[0,\bar{T}]$ which is subdivided in $N$ subintervals. For a fair comparison, we select $N=28$, which gives 87 variables, and we get 1.1919 as the computed bound on the minimum dwell-time. The computation time is 1.2 seconds. Thus we can see that, on this example, the SOS approach perform better with a comparable number of variables. Note, moreover, that the number of constraints involved in the discretization is larger than the one considered in the SOS program as well.
\end{example}

\subsection{A robustness result}

All the previous results can be robustified to account for parametric uncertainties affecting $A$ and $J$. To this aim, let us consider now that the matrices of the system~\eqref{eq:impsyst} are uncertain, possibly time-varying, and belonging to the following polytopes
\begin{equation}\label{eq:uncmat}
  A\in\mathcal{A}:=\co\left\{A_1,\ldots,A_N\right\}\ \text{and}\   J\in\mathcal{J}:=\co\left\{J_1,\ldots,J_N\right\}
\end{equation}
where $\co\{\cdot\}$ is the convex-hull operator. Before stating the main results, it is necessary to introduce the state-transition matrix $\Phi(\cdot)$, which corresponds to system~\eqref{eq:impsyst}-\eqref{eq:uncmat}, as
\begin{equation}\label{eq:evol}
      \dfrac{\d\Phi(s)}{\d s}=\left(\sum_{i=1}^N\lambda_i(s)A_i\right)\Phi(s),\ \Phi(0)=I
\end{equation}
where $\lambda(s)\in\Lambda_N:=\left\{\xi\in\mathbb{R}^N_{\ge0}:\ ||\xi||_1=1\right\}$ is sufficiently regular so that  solutions to \eqref{eq:evol} are well-defined, e.g. in a Carath\'{e}odory sense.  Associated to this transition matrix, we define the set $\boldsymbol{\Phi}_{\bar{T}}$ as
\begin{equation}
 \boldsymbol{\Phi}_{\bar{T}}:=\left\{\Phi(\bar{T}):\ \Phi(s)\ \text{solves\ \eqref{eq:evol}},\ \lambda(s)\in\Lambda_N,\ s\in[0,\bar{T}]\right\}.
\end{equation}
This set corresponds of all possible transition matrices $\Phi(\bar{T})$ obtained for all possible trajectories of the uncertain parameters $\lambda$. Note that the set $\boldsymbol{\Phi}_{\bar{T}}$ is strongly nonconvex and is difficult to compute exactly. This intricate structure illustrates the inherent difficulty in considering uncertain systems in a discrete-time setting. By reformulating the discrete-time conditions in terms of conditions \eqref{eq:c1} and \eqref{eq:c2}, this difficulty is circumvented and discrete-time stability results can be efficiently robustified. For conciseness, only the robustification of Theorem \ref{th:imp_p} will be discussed. Robust versions of Theorems \ref{th:imp_a} and \ref{th:imp_dt} can be obtained in the same way.
\begin{theorem}[Periodic impulses]\label{th:imp_p_rob}
 Let us consider the uncertain (time-varying)  impulsive system \eqref{eq:impsyst}-\eqref{eq:uncmat} with $\bar{T}$-periodic impulses, i.e. $T_k=\bar{T}$, $k\in\mathbb{N}$. Then, the following statements are equivalent:
 \begin{enumerate}
   \item The uncertain (time-varying)  impulsive system~\eqref{eq:impsyst}-\eqref{eq:uncmat} with $\bar{T}$-periodic impulses is quadratically stable\footnote{Quadratic stability of a linear uncertain system is defined here through the existence of a common quadratic Lyapunov function (i.e. independent of $\lambda$ in the present case) for the uncertain system; see e.g. \cite{Khargo:90}.}.
   \item There exists a matrix $P\in\mathbb{S}_{\succ0}^n$ such that the LMI
    \begin{equation}\label{eq:dkslddjsldjsldsj}
       J^T\Psi^TP\Psi J-P\prec0
    \end{equation}
    holds for all $(\Psi,J)\in\boldsymbol{\Phi}_{\bar{T}}\times\mathcal{J}$. Equivalently, the quadratic form $V(x)=x^TPx$ is a discrete-time Lyapunov function for the uncertain time-varying discrete-time system $z_{k+1}=\Psi_kJz_k$, for all $(\Psi_k,J)\in\boldsymbol{\Phi}_{\bar{T}}\times\mathcal{J}$.
   %
   \item There exist a differentiable matrix function $R:[0,\bar{T}]\mapsto\mathbb{S}^n$, $R(0)\succ0$, and a scalar $\eps>0$ such that the LMIs
  \begin{equation}
    A_i^TR(\tau)+R(\tau)A_i+\dot{R}(\tau)\preceq0
  \end{equation}
  and
  \begin{equation}
    J_i^TR(0)J_i-R(\bar{T})+\eps I\preceq0
  \end{equation}
  hold for all $\tau\in[0,\bar{T}]$ and all $i=1,\ldots,N$.
  \item There exist a differentiable matrix function $S:[0,\bar{T}]\mapsto\mathbb{S}^n$, $S(\bar{T})\succ0$, and a scalar $\eps>0$ such that the LMIs
  \begin{equation}
    A_i^TS(\tau)+S(\tau)A_i-\dot{S}(\tau)\preceq0
  \end{equation}
  and
  \begin{equation}
    J_i^TS(\bar{T})J_i-S(0)+\eps I\preceq0
  \end{equation}
  hold for all $\tau\in[0,\bar{T}]$ and all $i=1,\ldots,N$.
 \end{enumerate}
\end{theorem}
\begin{proof}
The proof simply follows from some convexity arguments.
\end{proof}

\section{Stabilization of periodic and aperiodic impulsive systems}\label{sec:stabz_is}

It is now shown that, unlike using looped-functionals, the current framework can be efficiently and accurately used for control design. To this aim, let us consider the impulsive system
\begin{equation}\label{eq:impsystc}
  \begin{array}{lcl}
    \dot{x}(t)&=&Ax(t)+B_cu_c(t),\ t\ne t_k\\
    x(t)&=&Jx^-(t)+B_du_d(t),\ t=t_k
  \end{array}
\end{equation}
where $u_c\in\mathbb{R}^{m_c}$ and $u_d\in\mathbb{R}^{m_d}$ are the control inputs. The following class of state-feedback control laws is considered:
\begin{equation}\label{eq:impsf}
  \begin{array}{rcl}
    u_c(t_k+\tau)&=&K_c(\tau)x(t_k+\tau),\ \tau\in[0,T_k),\\\
    u_d(t_k)&=&K_dx^-(t_k)
  \end{array}
\end{equation}
where the continuous control law is time-varying and the discrete one is time-invariant.
The purpose of this section is therefore to provide tractable conditions for finding suitable ${K_c:[0,\bar{T})\mapsto\mathbb{R}^{m_c\times n}}$ and $K_d\in\mathbb{R}^{m_d\times n}$ such that the closed-loop system \eqref{eq:impsystc}-\eqref{eq:impsf} is asymptotically stable.

\subsection{Periodic impulses case}

The next result gives constructive conditions for designing a control law of the form  \eqref{eq:impsf} for impulsive systems with $\bar{T}$-periodic impulses, i.e. $T_k=\bar{T}$. Suitable controller gains can, indeed, be directly extracted from the solutions of the sum-of-squares feasibility problems stated in the following result:
\begin{theorem}[Periodic impulses]\label{th:imp_p_stabz}
 The following statements are equivalent:
 \begin{enumerate}
   \item  There exists a control law \eqref{eq:impsf} such that the impulsive system  (\ref{eq:impsystc})-\eqref{eq:impsf} with $\bar{T}$-periodic impulses is asymptotically stable.
   %
%
   \item There exist a differentiable matrix function $S:[0,\bar{T}]\mapsto\mathbb{S}^n$, $S(0)\succ0$, a matrix function ${U_c:[0,\bar{T}]\mapsto\mathbb{R}^{m_c\times n}}$, a matrix $U_d\in\mathbb{R}^{m_d\times n}$ and a scalar $\eps>0$ such that the LMIs
  \begin{equation}\label{eq:c1z}
    \He[AS(\tau)+B_cU_c(\tau)]+\dot{S}(\tau)\preceq0
  \end{equation}
  and
  \begin{equation}\label{eq:c2z}
  \begin{bmatrix}
    -S(\bar{T})+\eps I & JS(0)+B_dU_d\\
    \star & -S(0)
  \end{bmatrix}\preceq0
  \end{equation}
  hold for all $\tau\in[0,\bar{T}]$. In such a case, suitable matrices for the control law \eqref{eq:impsf} are given by the expressions
  \begin{equation}
    \begin{array}{lclclcl}
      K_c(\tau)&=&U_c(\tau)S(\tau)^{-1},&&    K_d&=&U_dS(0)^{-1}.
    \end{array}
  \end{equation}
 \end{enumerate}
\end{theorem}
\begin{proof}
What has to be proven is the exactness of the stabilization conditions \eqref{eq:c1z}-\eqref{eq:c2z}.  By performing a congruence transformation on \eqref{eq:c1z} with respect to $\tilde{S}:=S^{-1}$ we get that
\begin{equation}\label{eq:kdslkdlsk}
    \He\left[\tilde{S}(\tau)(A+B_cK_c(\tau))\right]-\dot{\tilde{S}}(\tau)\preceq0
\end{equation}
where we used the facts that $K_c(\tau)=U_c(\tau)\tilde{S}(\tau)$ and $\tilde{S}(\tau)\dot{S}(\tau)\tilde{S}(\tau)=-\dot{\tilde{S}}(\tau)$. Looking now at the LMI \eqref{eq:c1z}, we can easily see that it is equivalent to
\begin{equation}
\begin{bmatrix}
  I\\
  (J+B_dK_d)^T
\end{bmatrix}^T \begin{bmatrix}
  -S(\bar{T}) & 0\\
  0 & S(0)
\end{bmatrix}\begin{bmatrix}
  I\\
  (J+B_dK_d)^T
\end{bmatrix}\prec0.
\end{equation}
Noting then that the central matrix has $n$ positive and $n$ negative eigenvalues, and that the outer-factors are of rank $n$, then the dualization lemma \citep{Scherer:00} applies, and we get the equivalent LMI
\begin{equation}\label{eq:kdslkdlsk768}
  (J+B_dK_d)^T\tilde{S}(\bar{T})(J+B_dK_d)-\tilde{S}(0)\prec0.
\end{equation}
Noting finally that the conditions \eqref{eq:kdslkdlsk}-\eqref{eq:kdslkdlsk768} are identical to \eqref{eq:c1b}-\eqref{eq:c2b} proves the result. Equivalence follows from the losslessness of the manipulations. 
%
\end{proof}
\begin{remark}
  Note that if the conditions of statement d) of Theorem \ref{th:imp_p} had been used, we would have obtained a controller matrix depending on the dwell-time $T_k$, which may not be implementable. This fact emphasizes the importance of statement e) of Theorem \ref{th:imp_p}.
\end{remark}

\subsection{Stabilization under minimum dwell-time}

The stabilization under minimum dwell-time is slightly more complicated since the controller gain $K_c(\tau)$ in \eqref{eq:impsf} must remain bounded as $\tau\to\infty$. In the best case, it should converge to a finite value. A way for solving this difficulty is to consider the following controller gain
\begin{equation}\label{eq:sfKdt}
  K_c(\tau)=\left\{\begin{array}{lcl}
        \tilde{K}_c(\tau)&&\text{if\ }\tau\in[0, \bar{T})\\
        \tilde{K}_c(\bar{T}) &&\text{if\ }\tau\in[\bar{T}, T_k)
  \end{array}\right.
\end{equation}
where $T_k\ge\bar{T}$, $k\in\mathbb{N}$ and $\tilde{K}_c(\tau)$ is some matrix function to be determined. This specific structure for the control law, as it will be emphasized later, arises naturally from the structure of the minimum dwell-time stability conditions and will be shown to be non-restrictive. Again the matrices of the controller can be extracted from the solutions of the feasibility problem stated in the following result:
\begin{theorem}[Minimum dwell-time]\label{th:imp_dt_stabz}
 The following statements are equivalent:
 \begin{enumerate}
   \item  There exist matrices $P\in\mathbb{S}_{\succ0}^n$, $K_d\in\mathbb{R}^{m_d\times n}$ and a matrix function $\tilde{K}_c:[0,\bar{T}]\mapsto\mathbb{R}^{m_c\times n}$ such that the matrix inequalities
   \begin{equation}\label{eq:stabzmDT1}
     (A+B_c\tilde{K}_c(\bar{T}))^TP+P(A+B_c\tilde{K}_c(\bar{T}))\prec0
   \end{equation}
   and
   \begin{equation}\label{eq:stabzmDT2}
     (J+B_dK_d)^T\Phi(\bar{T})^TP\Phi(\bar{T})(J+B_dK_d)-P\prec0
   \end{equation}
   holds where $\Phi:[0,\infty)\mapsto\mathbb{R}^{n\times n}$ is the transition matrix defined as
\begin{equation}
  \begin{array}{rcl}
    \dfrac{d}{d\tau}\Phi(\tau)&=&[A+B_cK_c(\tau)]\Phi(\tau),\ \tau\ge0\\
    \Phi(0)&=&I.
  \end{array}
\end{equation}
   \item There exist a differentiable matrix function $S:[0,\bar{T}]\mapsto\mathbb{S}^n$, $S(\bar{T})\succ0$, a matrix function ${U_c:[0,\bar{T}]\mapsto\mathbb{R}^{m_c\times n}}$, a matrix $U_d\in\mathbb{R}^{m_d\times n}$ and a scalar $\eps>0$ such that the LMIs
 \begin{equation}\label{eq:dtz1}
   \He[AS(\bar{T})+B_cU_c(\bar{T})]\prec0,
 \end{equation}
  \begin{equation}\label{eq:dtz2}
    \He[AS(\tau)+B_cU_c(\tau)]+\dot{S}(\tau)\preceq0
  \end{equation}
  and
  \begin{equation}\label{eq:dtz3}
  \begin{bmatrix}
    -S(0)+\eps I & JS(\bar{T})+B_dU_d\\
    \star & -S(\bar{T})
  \end{bmatrix}\preceq0
  \end{equation}
  hold for all $\tau\in[0,\bar{T}]$. In this case, suitable controller gains are retrieved using
  \begin{equation}
    \begin{array}{lclclcl}
      \tilde{K}_c(\tau)&=&U_c(\tau)S(\tau)^{-1},&& K_d&=&U_dS(\bar{T})^{-1}.
    \end{array}
  \end{equation}
 \end{enumerate}
 Moreover, in such a case, the closed-loop system  \eqref{eq:impsystc}-\eqref{eq:impsf}-\eqref{eq:sfKdt} is asymptotically stable with minimum dwell-time $\bar{T}$.
\end{theorem}
\begin{proof}
The first thing that has to be proven is the fact that statement (a) implies that the closed-loop system is stable with minimum dwell-time $\bar{T}$. The equivalence between (a) and (b) follows from Theorem \ref{th:imp_p_stabz} and the changes of variables $U_c(\tau)=\tilde{K}_c(\tau)S(\tau)$ and $U_d=K_dS(\bar{T})$.
Let us prove then that statement (a) implies that the closed-loop system is stable with minimum dwell-time $\bar{T}$. Two possible scenarios: 1) either impulses always arrive in finite-time, i.e. $\bar{T}\le T_k<\infty$; or 2) impulses stop at some point. Stability of the second case is straightforward from condition \eqref{eq:dtz1}. Let us then focus on the first case. We need to show there that the Lyapunov function $V(x)=x^TPx$ evaluated at times $t_k$ and along the trajectories of the closed-loop system \eqref{eq:impsystc}-\eqref{eq:impsf}-\eqref{eq:sfKdt} is pointwise decreasing and remains bounded between impulses. It is indeed pointwise decreasing whenever the LMI
\begin{equation}\label{eq:djksjdsjdkj}
     (J+B_dK_d)^T\Phi(\theta)^TP\Phi(\theta)(J+B_dK_d)-P\prec0
   \end{equation}
   holds for all finite $\theta\in[\bar{T},\infty)$. By using the fact that $\Phi(\bar{T}+\delta)=e^{(A+B_cK_c(\bar{T}))\delta}\Phi(\bar{T})$ for all $\delta\ge0$ and the same arguments as in the proof of Theorem \ref{th:imp_dt}, we have that conditions \eqref{eq:stabzmDT1} and \eqref{eq:stabzmDT2} implies that \eqref{eq:djksjdsjdkj} holds for all finite $\theta\in[\bar{T},\infty)$. The convergence of the trajectories of the impulsive system to 0 simply follows from the boundedness and continuity of the function $V(x(t))$ on every interval $(t_k,t_{k+1})$. The proof is complete.

\end{proof}

In the above result, we can clearly see that the structure of the control law fits exactly the structure of the conditions, which allows us to obtain lossless results. Without this maintain of the value of $\tilde{K}(\tau)$ to $\tilde{K}(\bar{T})$ for all $\tau\in[\bar{T}, T_k)$, deriving a minimum dwell-time stabilization result would have been much trickier. As a concluding remark on stabilization under minimum dwell-time, we note that the proposed control law just needs to be computed offline and is easy to implement.

\begin{example}
Let us consider the system~\eqref{eq:impsystc} with matrices
\begin{equation}\label{eq:exol}
  A=\begin{bmatrix}
    1 &0\\1& 2
  \end{bmatrix},\ B=\begin{bmatrix}
    1\\0
  \end{bmatrix}\ \text{and}\ J=\begin{bmatrix}
    1 &1\\1& 3
  \end{bmatrix}.
\end{equation}
Note that this system has both unstable flow and jumps. If we therefore assume that $K_d=0$, then the system cannot be stabilized for arbitrary dwell-times since the jumps are destabilizing. Thus, we want to compute $\tilde{K}_c(\tau)$ such that the minimum dwell-time is, at most, $\bar{T}=0.1$. Invoking Theorem \ref{th:imp_dt_stabz}, statement (b), with polynomial matrices $U_c(\tau)$ and $S(\tau)$ of order 1, we find the controller
\begin{equation*}
  \tilde{K}_c(\tau)=\dfrac{1}{d(\tau)}\begin{bmatrix}
1.4750481+3.2714889\tau-41.011914\tau^2\\
3.9063911-1.6733059\tau-37.472443\tau^2
  \end{bmatrix}^T
\end{equation*}
where $d(\tau)=-0.19767438+0.78454217\tau+7.6562219\tau^2$. State-trajectories of the closed-loop system, for some randomly generated impulse-times satisfying the dwell-time constraint, are depicted in Fig. \ref{fig:impcl}. We can clearly see that the controller stabilizes the system. 
In terms of computational complexity, only 27 variables are defined by SOSTOOLS.
\end{example}

\begin{figure}[h]
  \centering
  \includegraphics[width=0.4\textwidth]{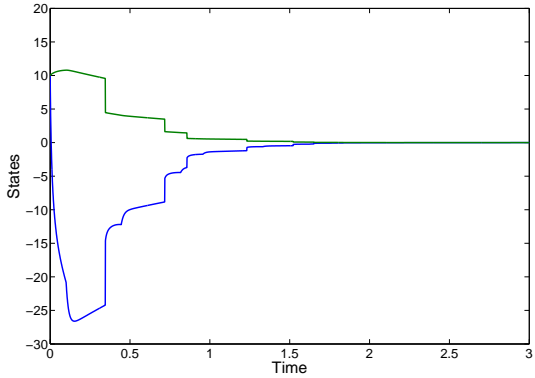}
  \caption{State-trajectories of the closed-loop system~\eqref{eq:exol}.}\label{fig:impcl}
\end{figure}

\section{Application to aperiodic uncertain sampled-data systems}\label{sec:sds}

It is well-known that sampled-data systems can be reformulated as impulsive systems. On the basis of this reformulation, all the results developed in the previous sections therefore apply. It is hence possible to obtain robust stabilization results for aperiodic uncertain sampled-data systems with inter-sampling times in a range, providing then a solution to this challenging problem. In terms of robust stability analysis, the proposed approach is computationally less expensive than those based on looped-functionals; see \citep{Briat:12h,Seuret:13}. 


\subsection{Preliminaries}

Let us consider here the continuous-time system
\begin{equation}\label{eq:sd}
    \dot{x}(t)=Ax(t)+Bu(t)
\end{equation}
where $x\in\mathbb{R}^n$ and $u\in\mathbb{R}^m$ are the state of the system and the control input, respectively. The control input is assumed to be computed from a sampled-data state-feedback control law given by
  \begin{equation}\label{eq:sd_sf}
    u(t)=K_1x(t_k)+K_2u(t_{k-1}),\ t\in[t_k,t_{k+1})
  \end{equation}
  where $K_1\in\mathbb{R}^{m\times n}$ and $K_2\in\mathbb{R}^{m\times m}$ are the control gains to be determined. Note that the control-law, viewed as a discrete-map from $x$ to $u$, is BIBO-stable if and only if $K_2$ is Schur. 
Above, the sequence of sampling instants $\{t_k\}_{k\in\mathbb{N}}$ is assumed to be strictly increasing and unbounded, i.e. $t_k\rightarrow\infty$, excluding therefore any Zeno behavior.

The sampled-data system~\eqref{eq:sd}-\eqref{eq:sd_sf} can be equivalently reformulated as the following impulsive system
\begin{equation}\label{eq:sd_sf_i}
    \begin{array}{rcl}
      \begin{bmatrix}
        \dot{x}(t)\\
        \dot{z}(t)
      \end{bmatrix}&=&\underbrace{\begin{bmatrix}
        A & B\\
        0 & 0
      \end{bmatrix}}_{\mbox{$\bar{A}$}}\begin{bmatrix}
       x(t)\\
        z(t)
      \end{bmatrix},\ t\ne t_k\\
      \begin{bmatrix}
       x(t)\\
       z(t)
      \end{bmatrix}&=&\underbrace{\begin{bmatrix}
        I & 0\\
        K_1 & K_2
      \end{bmatrix}}_{\mbox{$\bar{J}$}}\begin{bmatrix}
       x^-(t)\\
        z^-(t)
      \end{bmatrix},\ t=t_k
    \end{array}
  \end{equation}
  where $z\in\mathbb{R}^m$ is an additional state containing the value of the held control input at any time, i.e. $z(t)=u(t_k)$, $t\in[t_k,t_{k+1})$. For convenience, we also decompose  $\bar{J}$ as  $\bar{J}=J_0+B_0K$ where
  \begin{equation}
  J_0=\begin{bmatrix}
    I & 0\\
    0 & 0
  \end{bmatrix},\ B_0=\begin{bmatrix}
    0\\
    I
  \end{bmatrix}\ \text{and}\ K=\begin{bmatrix}
    K_1 & K_2
  \end{bmatrix}.
  \end{equation}


\subsection{Stabilization of sampled-data systems}

Stabilization of sampled-data systems being the most interesting problem, we will therefore focus on the stabilization of aperiodic sampled-data systems. The periodic case is readily recovered by setting $T_{max}=T_{min}=\bar{T}$. As in the previous stabilization results, a suitable stabilizing controller gain $K$ can be extracted from the solutions of a feasibility problem.


\begin{theorem}[Aperiodic sampled-data systems]\label{th:sd_stabz_ap}
The following statements are equivalent:
\begin{enumerate}
\item There exists a control law of the form \eqref{eq:sd_sf} that quadratically stabilizes the system  \eqref{eq:sd} for any aperiodic sampling instant sequence $\{t_k\}$ such that $T_k\in[T_{min},T_{max}]$.
%
      %
  \item There exist a differentiable matrix function $R:[0,T_{max}]\mapsto\mathbb{S}^{n+m}$, $S(0)\succ0$, a matrix $Y\in\mathbb{R}^{m\times(n+m)}$ and a scalar $\eps>0$ such that the conditions
  \begin{equation}
    \bar{A}(\tau)S(\tau)+S(\tau)\bar{A}(\tau)^T+\dot{S}(\tau)\preceq0
  \end{equation}
  and
  \begin{equation}
   \begin{bmatrix}
    -S(\theta)+\eps I & J_0+B_0Y\\
    \star & -S(0)
  \end{bmatrix}\preceq0
  \end{equation}
  hold for all $\tau\in[0,T_{max}]$ and all $\theta\in[T_{min},T_{max}]$. Moreover, when this statement holds, a suitable stabilizing control gain can be obtained using the expression $K=YS(0)^{-1}$.
  \end{enumerate}
\end{theorem}

\begin{remark}\label{rk:K20}
Interestingly, it is also possible to impose $K_2=0$ without introducing any conservatism. Such a controller can indeed be designed by simply imposing the value 0 to $n\times m$ right-upper block of the matrix $S(0)$. The reason why this equality constraint is non-restrictive lies in the fact that $e^{\bar{A}\bar{T}}J$ is a block lower triangular matrix, and that it is well-known that stability of cascade systems (that are represented in terms of block triangular matrices) can be exactly characterized by block diagonal Lyapunov functions.
\end{remark}

\begin{example}\label{ex:sd1}
  Let us consider the sampled-data system~\eqref{eq:sd} with matrices
  \begin{equation}\label{eq:ex_sd1}
    A=\begin{bmatrix}
      0 &  1\\
      0 & -0.1
    \end{bmatrix}\ \text{and}\  B=\begin{bmatrix}
      0\\
      0.1
    \end{bmatrix}.
  \end{equation}

  \textbf{Fixed control law:} Assume first that the control law is given as in \citep{Branicky:00}  by $K_1=\begin{bmatrix}
    -3.75 & -11.5
  \end{bmatrix}$ and $K_2=0$. Results in the aperiodic case ($T_{min}$ has been set to 0.001) are summarized in Table \ref{tab:sd1a} together with some comparisons with previous ones based on functionals. The proposed approach yields results that are very close to the looped-functional approach developed in \cite{Seuret:13} together with a reduced computational complexity. 
  The semidefinite program generated by SOSTOOLS involves, when $R$ or $S$ is of degree 4, 192 variables whereas  the approach of \citep{Seuret:13} involves 1256 variables when using a polynomial of order 3. The execution time is about 1 second whereas it is approximately of 4.46 seconds for the looped functional approach of \citep{Seuret:13}.

%

 \textbf{Control design:} Assume now that the control gains $K_1$ and $K_2$ have to be determined such that the closed-loop system is stable for any inter-sampling times in $[T_{min},T_{max}]$. Applying then Theorem \ref{th:sd_stabz_ap}, we obtain the results gathered in Table \ref{tab:design} where, following Remark \ref{rk:K20}, $K_2=0$ has been imposed in the three last scenarios. We can see that the computed controllers involve numerical values with reasonable magnitude. 
 The involved number of variables for degrees of $R$ or $S$ equal to 2 and 3 are 201 and 432, respectively.
\end{example}


\begin{example}\label{ex:sd3}
  Let us consider the following sampled-data system~\eqref{eq:sd} with matrices
  \begin{equation}\label{eq:ex_sd2}
    A=\begin{bmatrix}
      0 &  1\\
      -2 & 0.1
    \end{bmatrix}\ \text{and}\  B=\begin{bmatrix}
      0\\
      1
    \end{bmatrix}
  \end{equation}
  borrowed from the time-delay system literature \citep{NicuGuAbd:03}. Assuming the control law $K_1=\begin{bmatrix}
    1 & 0
  \end{bmatrix}$ and $K_2=0$, we get the results of Table \ref{tab:sd1a}. 
  We obtain a result very close to the one of \citep{Seuret:13} using a matrix $R$ or $S$ of degree 6, which corresponds to 330 variables. The result of \citep{Seuret:13} considering a matrix polynomial of order 5 involves 3414 variables. The execution time is about 1.12 seconds whereas it is approximately of 15.34 seconds for the looped functional approach of \citep{Seuret:13}.
  %
\end{example}


\begin{table}
  \centering
  \caption{Estimates on the minimum and/or maximum sampling period for the systems of Examples \ref{ex:sd1} and \ref{ex:sd3} - Aperiodic case.}\label{tab:sd1a}
  \begin{tabular}{|c|c|c|c|c|}
    \hline
    \multirow{3}{*}{ } &  \multirow{2}{*}{$d_R$} & System \eqref{eq:ex_sd1} & \multicolumn{2}{c|}{System \eqref{eq:ex_sd2}}\\
    \cline{3-5}
     & & $T_{max}$ &  $T_{min}$ &$T_{max}$\\
      \hline
      \hline
     \multirow{2}{*}{Th. \ref{th:imp_p}}
      &4 & 1.7279  & 0.4 & 1.6316\\
      &6 & 1.7252  & 0.4 &1.8270\\
      \hline
     \citep{Fridman:04} & -- & 0.869 & --& --\\
     \citep{Naghshtabrizi:08} & -- &  1.113 & -- & --\\
      \citep{Fridman:10} & -- &  1.695  & --& --\\
      \citep{Liu:10} & -- & 1.695  & -- & --\\
      \citep{Seuret:12} & -- &  1.723  & 0.400 & 1.251\\
      \hline
       \multirow{2}{*}{\citep{Seuret:13}} & 3 & 1.7294  & 0.4 &1.820\\
       & 5 & 1.7294  & 0.4&1.828\\
      \hline
  \end{tabular}
\end{table}

\begin{table}
  \centering
  \caption{Control design results for system~\eqref{eq:ex_sd1} using Theorem \ref{th:sd_stabz_ap}}\label{tab:design}
  \begin{tabular}{|c|c||c|c|c|}
    \hline
   $T_{min}$ & $T_{max}$ & $K_1$ & $K_2$ & $d_R$\\
   \hline
   \hline
  \multirow{2}{*}{0.001}  & 10 & $\begin{bmatrix}
    -0.1145  & -0.8088
  \end{bmatrix}$ & -0.0024 & 2\\
  & 50 & $\begin{bmatrix}
 -0.0202  & -0.1560
  \end{bmatrix}$ & -0.0030 & 2\\
  \hline
    \multirow{2}{*}{0.001} & 10 & $\begin{bmatrix}
-0.0310  & -0.3222
  \end{bmatrix}$ & 0 & 3\\
  & 50 & $\begin{bmatrix}
 -0.0259 & -0.2726
  \end{bmatrix}$ & 0 & 4\\
  \hline
  \end{tabular}
\end{table}

\subsection{Robust stabilization of periodic and aperiodic sampled-data systems}

Results on robust stabilization of sampled-data systems are straightforward extensions of Theorem \ref{th:sd_stabz_ap}; they are omitted for brevity. Only the following example is discussed:
\begin{example}
  Let us consider the uncertain sampled-data system~\eqref{eq:sd} with matrices
\begin{equation}\label{eq:ex_sd3}
A\in\mathcal{A}=\co\left\{\begin{bmatrix}
    0 & 1\\
    0 & -0.1
  \end{bmatrix},\delta\begin{bmatrix}
    0 & 1\\
    0 & -0.1
  \end{bmatrix}\right\}\ \text{and}\ B=\begin{bmatrix}
    0\\
    1
  \end{bmatrix}
\end{equation}
where $\delta$ is a positive parameter. We then apply Theorem~\ref{th:sd_stabz_ap} to design robust state-feedback controllers for different values for $\delta>0$ and $T_{max}>0$. The results are summarized in Table \ref{tab:design2} where we can see that the system can be stabilized for a quite wide range of values for the parameter $\delta$ and the maximal sampling period $T_{max}$. 
For a polynomial or order 2, the semidefinite program involve 237 variables. The execution time including pre- and post-processing is about 2.23 seconds.

\begin{table}
  \centering
  \caption{Control design results for system~\eqref{eq:ex_sd3} using Theorem \ref{th:sd_stabz_ap}}\label{tab:design2}
  \begin{tabular}{|c|c|c||c|c|c|}
    \hline
   $\delta$ & $T_{min}$ & $T_{max}$ & $K_1$ & $K_2$ & $d_R$\\
   \hline
   \hline
   5 &\multirow{2}{*}{0.001} & 10 & $\begin{bmatrix}
   -0.0757 & -0.7306
  \end{bmatrix}$ & -0.0006 & 2\\
   5 & & 20 & $\begin{bmatrix}
    -0.0411  & -0.3835
  \end{bmatrix}$ & -0.0022 & 2\\
  \hline
20 &\multirow{2}{*}{0.001} & 10 & $\begin{bmatrix}
   -0.0578  & -0.5560
  \end{bmatrix}$ & -0.0025 & 2\\
  20 & & 20 & $\begin{bmatrix}
  -0.0339 & -0.3121
  \end{bmatrix}$ & -0.0019 & 2\\
  \hline
  \end{tabular}
\end{table}
%
\end{example}

%

\bibliographystyle{plainnat}

\end{document}